\documentclass[a4paper,11pt]{amsart}
\usepackage{mathrsfs}
\usepackage{amsfonts}
\usepackage{txfonts}
\usepackage{amssymb}
\usepackage[arrow,matrix]{xy}
\usepackage{amsmath,amssymb,amscd,bbm,amsthm,mathrsfs,dsfont}
\usepackage{amsmath,amscd}\input amssym.def
\usepackage{amsfonts,amssymb}
\newtheorem{theorem}{Theorem}[section]
\newtheorem{lemma}[theorem]{Lemma}
\theoremstyle{definition}
\newtheorem{prop}[theorem]{Proposition}
\newtheorem{defn}[theorem]{Definition}
\newtheorem{exa}[theorem]{Example}

\newtheorem{corollary}[theorem]{Corollary}
\theoremstyle{remark}
\newtheorem{rem}[theorem]{Remark}
\DeclareMathOperator{\Ext}{Ext} 
\DeclareMathOperator{\Tor}{Tor} \numberwithin{equation}{section}

\begin{document}

\title{Quasi-Koszulity and minimal Horseshoe Lemma}
\author{Jia-Feng L\"{U}}

\address{ Jia-Feng L\"{u}}
\address{Department of Mathematics, Zhejiang Normal University, Jinhua, Zhejiang, 321004 P.R. China}
\email{jiafenglv@gmail.com, jiafenglv@zjnu.edu.cn}
\thanks{Supported by National Natural Science Foundation of China (No. 11001245), Zhejiang
Province Department of Education Fund (No. Y201016432) and Zhejiang
Innovation Project (No. T200905).} 
\begin{abstract}
In this paper, the criteria for minimal Horseshoe Lemma to be true
are given via quasi-$\delta$-Koszul modules, which are the nongraded
version of $\delta$-Koszul modules first introduced by Green and
Marcos in 2005. Moreover, some applications of minimal Horseshoe
Lemma are also given.
\end{abstract}
\maketitle

\section{Introduction}
It is well-known that Horseshoe Lemma is a basic tool in the theory
of homological algebra, which provides a method to construct a
projective resolution for the middle term via the ones of the first
and the third terms of a given short exact sequence. But what
happens if we replace the projective resolutions by the minimal
projective resolutions? See some easy examples first:

\begin{enumerate}
\item Let $A=\Bbbk[x]$, a graded polynomial algebra, $M=A/(x^2)$,
$K=A/(x)[-1]$ and $N=\Bbbk$, a fixed field. Now under a routine
computation, we can get the following corresponding minimal
projective resolutions:
$$\xymatrix{
  0 \ar[r] & A[-2]\ar[r] & A[-1]\ar[r]  & K\ar[r] & 0,}$$
$$\xymatrix{
  0 \ar[r] & A[-2]\ar[r] & A\ar[r]  & M \ar[r] & 0}$$

and
$$\xymatrix{
  0 \ar[r] & A[-1]\ar[r] & A\ar[r]  & N \ar[r] & 0.}$$
Now it is clear that we have $A[-2]\ncong A[-2] \oplus A[-1]$ and
$A\ncong A\oplus A[-1]$ as graded $A$-modules, and the exact
sequence
$$\xymatrix{
  0 \ar[r] & K\ar[r] & M\ar[r]  & N \ar[r] & 0,}$$
where $[\;]$ denotes the shift functor given by $(M[n])_t=M_{n+t}$
for any $\mathbb{Z}$-graded module $M$ and $n,\;t\in\mathbb{Z}$.

\item Let $R$ be a semiperfect Noetherian ring with identity
(over which every finitely generated left module has a finitely
generated projective cover), $M$ a finitely generated $R$-module and
$Rad(M)$ the radical of $M$. Set $K=Rad(M)$ and $N=M/Rad(M)$.
Obviously, we have the following short exact sequence
$$\xymatrix{
  0 \ar[r] & K\ar[r] & M\ar[r]  & N \ar[r] & 0.}$$
Note that $R$ is semiperfect, thus all the finitely generated
$R$-modules possess projective covers.

Let $\xymatrix{
  P_0 \ar[r] & K\ar[r] &  0,}\;\;\xymatrix{
  L_0 \ar[r] & M\ar[r] &  0} \;\;\rm{and}\;\;\xymatrix{
  Q_0 \ar[r] & N\ar[r] &  0}$ be the corresponding projective
covers. Then $L_0\cong Q_0$ as $R$-modules since $N=M/Rad(M)$.
Therefore, we have $L_0\ncong P_0\oplus Q_0$ as $R$-modules since
$P_0\neq0$.

\item Let $A$ be a $\delta$-Koszul algebra (Green and Marcos, 2005) and
$\mathcal{K}^{\delta}(A)$ be the category of $\delta$-Koszul
modules. Let
$$\xymatrix{
  0 \ar[r] & K\ar[r] & M\ar[r]  & N \ar[r] & 0}$$ be an exact
sequence in $\mathcal{K}^{\delta}(A)$, and $\xymatrix{
  \mathcal{P}_* \ar[r] & K\ar[r] &  0,}\;\;\xymatrix{
  \mathcal{L}_* \ar[r] & M\ar[r] &  0}$ and $\xymatrix{
  \mathcal{Q}_* \ar[r] & N\ar[r] &  0}$ be the corresponding minimal graded projective resolutions. Then by Theorem 2.6 of
\cite{L} (also see the below), we have the following commutative
diagram with exact rows and columns
$$\xymatrix{
0 \ar[r] &\mathcal{P}_* \ar[d] \ar[r] & \mathcal{L}_* \ar[d] \ar[r] & \mathcal{Q}_* \ar[d] \ar[r] & 0 \\
  0 \ar[r] &K \ar[d] \ar[r] & M \ar[d] \ar[r] & N \ar[d] \ar[r] & 0 \\
 & 0  & 0 & 0 }$$$$\rm{(Fig. \;1.1)}$$ and $L_n\cong P_n\oplus Q_n$ as graded $A$-modules for all $n\geq0$.
\end{enumerate}

From the above examples, we can see clearly that if we replace
projective resolutions by minimal projective resolutions in the
Horseshoe Lemma, the conclusion is inconclusive. For the convenience
of narrating, we state the so-called ``minimal Horseshoe Lemma''
now. Roughly speaking, minimal Horseshoe Lemma is the ``minimal''
version and a special case of the classic Horseshoe Lemma, which can
be stated as follows:

\begin{itemize}
\item Let $R$ be any ring with identity and $\xymatrix{0 \ar[r] & K
\ar[r] & M\ar[r]  & N \ar[r] & 0}$ be an exact sequence of
$R$-modules. Then for any given diagram
$$\xymatrix{
& \mathcal{P}_* \ar[d]  &  & \mathcal{Q}_* \ar[d] \\
  0 \ar[r] &K \ar[d] \ar[r] & M  \ar[r] & N \ar[d] \ar[r] & 0 \\
 & 0  &  & 0 }$$$$\rm{(Fig. \;1.2)}$$
with $\mathcal{P}_*$ and $\mathcal{Q}_*$ being minimal projective
resolutions of $K$ and $N$, respectively. Then we can complete Fig.
1.2 into Fig. 1.1 such that the rows and columns in Fig. 1.1 are all
exact and $\xymatrix{ \mathcal{L}_* \ar[r] & M \ar[r] & 0
  }$ is also a minimal projective resolution.
\end{itemize}

Therefore, it is interesting and meaningful to find conditions for
the minimal Horseshoe Lemma to be true. In 2008, Wang and Li studied
the conditions for the minimal Horseshoe Lemma to be true in the
graded case and gave some sufficient conditions. Moreover, they said
`Though we have found some sufficient conditions for the minimal
Horseshoe Lemma to be held, an interesting but difficult question is
how to find some necessary conditions''. In fact, Theorem 2.6 of
\cite{L} has provided a necessary and sufficient condition for the
minimal Horseshoe Lemma to be true via $\delta$-Koszul modules in
the graded case:

\begin{itemize}
\item  ({\bf Theorem 2.6, \cite{L}}) Let $A$ be a
standard graded algebra and $$\xymatrix{
  0 \ar[r] & K\ar[r] & M\ar[r]  & N \ar[r] & 0}$$ be an exact sequence with $M,\;N$ being $\delta$-Koszul modules.
Then $K$ is a $\delta$-Koszul module if and only if the minimal
Horseshoe Lemma holds, here we refer to Section 2 (or \cite{L} and
\cite{GM}) for the notions of standard graded algebra and
$\delta$-Koszul module.
\end{itemize}

As direct corollaries, we can obtain necessary and sufficient
conditions for the minimal Horseshoe Lemma to be true via Koszul
(see \cite{P}), $d$-Koszul (see \cite{B}, \cite{GMMZ} and \cite{YZ})
and piecewise-Koszul (see \cite{LHL}) objects and so on since all of
them are special $\delta$-Koszul objects. Recently, Green and
Mart\'{\i}nez-Villa generalized Koszul objects to the nongraded case
and introduced quasi-Koszul objects (see \cite{GM1}); He, Ye and Si
generalized $d$-Koszul objects to the nongraded case and introduced
quasi-$d$-Koszul objects (see \cite{HY} and \cite{S}) and the author
of the present paper generalized piecewise-Koszul objects to the
nongraded case and introduced quasi-piecewise-Koszul objects (see
\cite{L1} and \cite{LPW}). Motivated by the above, now one can ask a
natural question: Can we give some conditions for the minimal
Horseshoe Lemma to be true via these ``quasi-Koszul-type'' objects?

The main purpose of this paper is to give an answer to the above
question and we prove the following result:

\medskip
\noindent{{\bf{Theorem A}}} {\it Let $R$ be an augmented Noetherian
semiperfect algebra with Jacobson radical $J$ and
$$\xymatrix{\xi:
  0 \ar[r] & K \ar[r] & M\ar[r]  & N \ar[r] & 0}$$ be a short exact sequence
in the category of quasi-$\delta$-Koszul modules. Then $JK=K\cap JM$
if and only if the minimal Horseshoe Lemma holds with respect to
$\xi$.}

As an immediate corollary of Theorem A, we obtain the following
results:

\medskip
\noindent{{\bf{Corollary B}}} {\it Let $R$ be an augmented
Noetherian semiperfect algebra with Jacobson radical $J$ and
$$\xymatrix{\xi:
  0 \ar[r] & K \ar[r] & M\ar[r]  & N \ar[r] & 0}$$ be a short exact sequence
in the category $\mathcal{C}$. Then the following statements are
true:
\begin{enumerate}
\item If $\mathcal{C}$ denotes the category of quasi-Koszul modules,
then $JK=K\cap JM$ if and only if the minimal Horseshoe Lemma holds
with respect to $\xi$.
\item If $\mathcal{C}$ denotes the category of quasi-$d$-Koszul modules,
then $JK=K\cap JM$ if and only if the minimal Horseshoe Lemma holds
with respect to $\xi$.
\item If $\mathcal{C}$ denotes the category of quasi-piecewise-Koszul modules,
then $JK=K\cap JM$ if and only if the minimal Horseshoe Lemma holds
with respect to $\xi$.
\end{enumerate}}
\begin{rem}
In Corollary B, (1) and (2) show that Theorem 2.8 of \cite{WL} and
Theorem 3.1 of \cite{LZ} are in fact necessary and sufficient
conditions; and (3) has been appeared and proved directly in
\cite{LPW}.
\end{rem}

With the help of minimal Horseshoe Lemma, one can obtain some
surprising results which may be wrong in general:

\medskip
\noindent{{\bf{Theorem C}}} {\it Let $R$ be an augmented Noetherian
semiperfect algebra with Jacobson radical $J$ and
$$\xymatrix{\xi:
  0 \ar[r] & K \ar[r] & M\ar[r]  & N \ar[r] & 0}$$ be a short exact sequence in
the category of finitely generated $R$-modules. If the minimal
Horseshoe Lemma holds for $\xi$, then we have the following
statements:
\begin{enumerate}
\item $M$ is
projective if and only if $K$ and $N$ are both projective;

\item $pd(M)=\max\{pd(K),
pd(N)\}$.\end{enumerate}}

As mentioned above, the notion of {\it quasi-Koszul module} was
introduced by Green and Mart\'{\i}nez-Villa in 1996 (see
\cite{GM1}). Moreover, they studied the extension closure of the
category of quasi-Koszul modules and got the following result:

\begin{itemize}
\item Let $R$ be a Noetherian semiperfect algebra with Jacobson radical $J$ and $$\xymatrix{
  0 \ar[r] & K\ar[r] & M\ar[r]  & N \ar[r] & 0}$$ be an exact sequence of finitely generated
$R$-modules with $JK=K\cap JM$. If $K$ and $N$ are quasi-Koszul
modules, then so is $M$.
\end{itemize}

Motivated by the above, a naive but interesting question is: If $M$
and $N$ are quasi-Koszul, then is $K$ quasi-Koszul or if $K$ and $M$
are quasi-Koszul, then is $N$ quasi-Koszul? Green and
Mart\'{\i}nez-Villa did not discuss these in \cite{GM1}. With the
help of minimal Horseshoe Lemma, we get the following assertions:

\medskip
\noindent{{\bf{Theorem D}}} {\it Let $R$ be an augmented Noetherian
semiperfect algebra with Jacobson radical $J$ and
$$\xymatrix{\xi:
  0 \ar[r] & K \ar[r] & M\ar[r]  & N \ar[r] & 0}$$ be a short exact sequence in
the category of finitely generated $R$-modules with the minimal
Horseshoe Lemma holding for $\xi$. Then we have the following
statements:
\begin{enumerate}
\item If $M$ is a quasi-Koszul module, then so is $K$;

\item If we have $J^2\Omega ^i(K) = \Omega^i(K)
\cap J^2\Omega^i(M)$ for all $i\geq0$, then $N$ is a quasi-Koszul
module provided that $K$ and $M$ are quasi-Koszul
modules.\end{enumerate}}

In a word, the main purposes of this paper are to find some
equivalent conditions and applications for minimal Horseshoe Lemma.
More precisely, in Section 2, as preknowledge, we will give the
definition of quasi-$\delta$-Koszul modules. In Section 3, we will
prove Theorem A. Section 4 mainly focus on the applications of
minimal Horseshoe Lemma and we will prove Theorems C and D.

\medskip
\section{Quasi-$\delta$-Koszul modules}
In this section, $A=\bigoplus_{i\geq 0}A_i$ denotes a standard
graded algebra, i.e., $A$ satisfies (a)
$A_0=\Bbbk\times\cdots\times\Bbbk$, a finite product of the ground
field $\Bbbk$; (b) $A_i \cdot A_j = A_{i+j}$ for all $0\leq i,
j<\infty$; and (c) dim$_{\Bbbk}A_i<\infty$ for all $i\geq 0$.
Clearly, the graded Jacobson radical of a standard graded algebra
$A$ is obvious $\bigoplus_{i\geq 1}A_i$, which is usually denoted by
$J$.

From \cite{GM1}, we know that standard graded algebras can be
realized by finite quivers:
\begin{prop}
Let $A$ be a standard graded algebra. Then there exists a finite
quiver $\Gamma$ and a graded ideal $I$ in $\Bbbk\Gamma$ with
$I\subset \sum_{n\geq 2}(\Bbbk\Gamma)_n$ such that $A\cong
\Bbbk\Gamma/I$ as graded algebras.
\end{prop}
\begin{defn}\label{d1}
Let $A$ be a standard graded $\Bbbk$-algebra and $M=\bigoplus_{i\geq
0}M_i$ a finitely generated graded $A$-module. We call $M$ a
{\it{$\delta$-Koszul module\/}} provided that $M$ admits a minimal
graded projective resolution
$$\xymatrix{
  \cdots \ar[r] & P_n\ar[r] & P_{n-1} \ar[r]& \cdots \ar[r] & P_1 \ar[r] & P_0 \ar[r] & M \ar[r] &
  0,}$$ such that each $P_n$
is generated in degree $\delta(n)$ for all $n\geq0$, where $\delta:
\mathbb{N}\rightarrow\mathbb{N}$ is a set function and $\mathbb{N}$
denotes the set of natural numbers.

In particular, the standard graded algebra $A$ will be called a {\it
$\delta$-Koszul algebra} if the trivial $A$-module $A_0$ is a
$\delta$-Koszul module.
\end{defn}
\begin{rem}
(1) The set function $\delta$ is in fact strictly increasing.

(2) The notion of $\delta$-Koszul algebra in this paper is different
from its original definition (\cite{GM}) and we don't request its
Yoneda algebra to be finitely generated.

\end{rem}
\begin{exa}
(1) Koszul algebras/modules (see \cite{P}) are $\delta$-Koszul
algebras/modules, where the set function $\delta(i)=i$ for all
$i\geq 0$;

(2) $d$-Koszul algebras/modules (see \cite{B} and \cite{GMMZ}) are
$\delta$-Koszul algebras/modules, where the set function
$$\delta(n)=\left\{\begin{array}{ll}
\frac{nd}{2},  &  \mbox{if $n$ is even,}\\
\frac{(n-1)d}{2}+1,  &  \mbox{if $n$ is odd.}
\end{array}
\right.
$$

(3) Piecewise-Koszul algebras/modules (see \cite{LHL}) are
$\delta$-Koszul algebras/modules, where the set function
$$\delta(n)=\left\{\begin{array}{llll}
\frac{nd}{p},  &  \mbox{if $n \equiv 0 (\textrm{mod} p)$,}\\
\frac{(n-1)d}{p}+1,  &  \mbox{if $n \equiv 1 (\textrm{mod} p)$,}\\
\cdots& \cdots, \\
\frac{(n\!-\!p+1)d}{p}+p\!-\!1, \quad & \mbox{if $n\equiv p\!-\!\!1
(\textrm{mod} p)$.}
\end{array}
\right.
$$
and $d\geq p\ge 2$ are given integers.
\end{exa}

The following theorem generalizes (Proposition 3.1, \cite{GM1}).
\begin{theorem}\label{th1}
Let $A=\Bbbk\Gamma/I$ be a standard graded algebra and
$$\xymatrix{
  \cdots \ar[r] & P_n \ar[r]^{d_n} & \cdots \ar[r] & P_1 \ar[r]^{d_1} & P_0 \ar[r]^{d_0} & A_0 \ar[r] &
  0}$$ a minimal graded projective resolution of the trivial
$A$-module $A_0$. Then the following statements are equivalent:
\begin{enumerate}
\item $A$ is a $\delta$-Koszul algebra;

\item for all $n\geq 0$, $\ker d_n\subseteq
J^{\delta(n+1)-\delta(n)}P_n$ and $J\ker d_n=\ker d_n\cap
J^{\delta(n+1)-\delta(n)+1}P_n$;

\item for any fixed $n\geq 1$ and $1\leq i\leq n$,
$P_i=\bigoplus_{l\geq 1}Ae_{i_l}[-\delta(i)]$, the component of
$d_i(e_{i_l})$ in some $Ae_{i-1_{m}}$ is in
$A_{\delta(i)-\delta(i-1)}$, $\ker d_n\subseteq
J^{\delta(n+1)-\delta(n)}P_n$ and $J\ker d_n=\ker d_n\cap
J^{\delta(n+1)-\delta(n)+1}P_n$. \end{enumerate}
\end{theorem}
\begin{proof}
(1)$\Rightarrow$(2) Suppose that $A$ is a $\delta$-Koszul algebra.
Then for all $n\geq 0$, $P_n$ is generated in degree $\delta(n)$.
Note that $d_{n+1}(P_{n+1})=\ker d_{n}$, which implies that $\ker
d_n$ is generated in degree $\delta(n+1)$. But recall that $P_n$ is
generated in degree $\delta(n)$, hence the elements of degree
$\delta(n+1)$ of $P_n$ are in $J^{\delta(n+1)-\delta(n)}P_n$. Thus
for all $n\geq 0$, $\ker d_n\subseteq J^{\delta(n+1)-\delta(n)}P_n$.
Now it is clear that $J\ker d_n\subseteq\ker d_n\cap
J^{\delta(n+1)-\delta(n)+1}P_n$. Now let $x\in\ker d_n\cap
J^{\delta(n+1)-\delta(n)+1}P_n$ be a homogeneous element of degree
$i$. It is easy to see that $i\geq\delta(n+1)+1$. If $x$ is not in
$J\ker d_n$, then $x$ is a generator of $\ker d_n$, which implies
that $\ker d_n$ is generated in degree larger than $\delta(n+1)+1$
since the degree of $x$ is larger than $\delta(n+1)+1$, which
contradicts to that $\ker d_n$ is generated in degree $\delta(n+1)$.
Therefore, $x\in J\ker d_n$ and $J\ker d_n\supseteq\ker d_n\cap
J^{\delta(n+1)-\delta(n)+1}P_n$. Thus we are done.

(2)$\Rightarrow$(1) First we claim that for all $n\geq 0$,
$(P_n)_j=0$ for all $j<\delta(n)$. Do it by induction on $n$. First
we prove that $(P_0)_j=0$ for $j< \delta(0)=0$. If not, since $P_0$
is a finitely generated graded module, there exists a smallest
$j_0<\delta(0)$ such that $(P_0)_{j_0}\neq 0$. Let $x\neq 0$ be a
homogeneous element of $P_0$ of degree $j_0$. Then $d_0(x)=0$ since
$d_0(x)\in (A_0)_{j_0}$ and $A_0=(A_0)_{0}$, which implies that
$x\in \ker d_0\subset JP_0$, which contradicts the choice of $j_0$.
Now suppose that $(P_{n-1})_j=0$ for all $j<\delta(n-1)$. Similarly,
assume that there exists a smallest $j_0'<\delta(n)$ such that
$(P_n)_{j_0'}\neq 0$. Let $x\neq 0$ be a homogeneous element of
$P_n$ of degree $j'_0$. Note that $d_n(x)\in Im d_n=\ker
d_{n-1}\subseteq J^{\delta(n)-\delta(n-1)}P_{n-1}$, we have
$d_n(x)=0$ since $J^{\delta(n)-\delta(n-1)}P_{n-1}$ is supported in
$\{i|i\geq \delta(n)\}$. Therefore, $x\in \ker d_n\subseteq
J^{\delta(n+1)-\delta(n)}P_n$, which contradicts the choice of
$j'_0$.

Now we claim that for any $x\in (P_n)_i$ with $i>\delta(n)$, then
$x\in J^sP_n$ for some $s>0$. If we prove this claim, then it is
clear that for all $n\geq 0$, $P_n$ is generated in degree
$\delta(n)$. In fact, we also prove this by induction on $n$. Note
that $A_0$ is generated in degree 0, thus $d_0(x)\in JA_0=J$, which
implies that $x\in d_0^{-1}(J)=JP_0+\ker d_0\subseteq JP_0$.
Therefore, $P_0$ is generated in degree 0. Suppose that for any
$x\in (P_{n-1})_i$ with $i>\delta(n-1)$, then we have $x\in
J^sP_{n-1}$ for some $s>0$ and $P_{n-1}$ is generated in degree
$\delta(n-1)$. By the condition $J\ker d_{n-1}=\ker d_{n-1}\cap
J^{\delta(n)-\delta(n-1)+1}P_{n-1}$, we have $\ker d_{n-1}$ is
generated in degree $\delta(n)$, which implies that $P_n$ is
generated in degree $\delta(n)$ for all $n\geq 0$. Of course, for
any $x\in (P_n)_i$ with $i>\delta(n)$, we have $x\in J^sP_n$ for
some $s>0$.

(1), (2)$\Rightarrow$(3) Suppose that $A$ is a $\delta$-Koszul
algebra. Then for all $i\geq 0$, $P_i$ is generated in degree
$\delta(i)$. Thus all $e_{i_l}$ are of degree $\delta(i)$, which
implies that $d_i(e_{i_l})\in (P_{i-1})_{\delta(i)}$. But $P_{i-1}$
is generated in degree $\delta(i-1)$, hence
$(P_{i-1})_{\delta(i)}\subseteq
A_{\delta(i)-\delta(i-1)}(P_{i-1})_{\delta(i-1)}$. Now (3) is clear
by (2).

(3)$\Rightarrow$(1) By an induction on $n$, it suffices to prove
that $P_0$ is generated in degree $\delta(0)$ and $\ker d_0$ is
generated in degree $\delta(1)$, which is similar to the proof of
$(2)\Rightarrow (1)$ and we omit the details.
\end{proof}

\begin{corollary}\label{cor1}
Let $A$ be a standard graded algebra, $M$ a finitely 0-generated
graded $A$-module and
$$\xymatrix{
  \cdots \ar[r] & P_n \ar[r]^{d_n} & \cdots \ar[r] & P_1 \ar[r]^{d_1} & P_0 \ar[r]^{d_0} & M \ar[r] &
  0}$$ a minimal graded projective resolution of $M$. Then $M$ is a
  $\delta$-Koszul module if and only if for all $n\geq 0$, $\ker d_n\subseteq
J^{\delta(n+1)-\delta(n)}P_n$ and $J\ker d_n=\ker d_n\cap
J^{\delta(n+1)-\delta(n)+1}P_n$.
\end{corollary}

Motivated by Corollary \ref{cor1}, we get the following definition:
\begin{defn}
Let $R$ be a Noetherian semiperfect algebra with Jacobson radical
$J$ and $M$ a finitely generated $R$-module. Let
$$\xymatrix{
  \cdots \ar[r] & P_n \ar[r]^{d_n} & \cdots \ar[r] & P_1 \ar[r]^{d_1} & P_0 \ar[r]^{d_0} & M \ar[r] &
  0}$$ be a minimal projective resolution of $M$. Then we call $M$ a
  {\it quasi-$\delta$-Koszul module} if for all $n\geq 0$, we have $\ker d_n\subseteq
J^{\delta(n+1)-\delta(n)}P_n$ and $J\ker d_n=\ker d_n\cap
J^{\delta(n+1)-\delta(n)+1}P_n$, where $\delta:
\mathbb{N}\rightarrow\mathbb{N}$ is a strictly increasing set
function.

In particular, $R$ is called a {\it quasi-$\delta$-Koszul algebra}
if $R/J$ is a quasi-$\delta$-Koszul module.

Let $\mathcal{Q}^{\delta}(R)$ denote the category of
quasi-$\delta$-Koszul modules.
\end{defn}

\begin{exa}
Quasi-Koszul algebras/modules (see \cite{GM1}), quasi-$d$-Koszul
algebras/modules (see \cite{HY}) and quasi-piecewise-Koszul
algebras/modules (see \cite{L}) are all special
quasi-$\delta$-Koszul modules.
\end{exa}

\medskip
\section{Criteria for minimal Horseshoe Lemma}
Throughout this section, $R$ denotes an augmented Noetherian
semiperfect algebra with Jacobson radical $J$ and we will mainly
concentrate on the proof of Theorem A.

\begin{lemma}\label{lem1}
Let $\xymatrix{
  0 \ar[r] & K \ar[r] & M\ar[r]  & N \ar[r] & 0  }$ be an exact
sequence of finitely generated $R$-modules. Then $JK=K\cap JM$ if
and only if we have the following commutative diagram with exact
rows and columns
$$\xymatrix{
           & 0 \ar[d]        & 0 \ar[d]        & 0 \ar[d] \\
  0 \ar[r] & \Omega^1(K) \ar[d] \ar[r] & \Omega^1(M) \ar[d]
  \ar[r] & \Omega^1(N) \ar[d] \ar[r] & 0 \\
  0 \ar[r] & P_0 \ar[d] \ar[r] & L_0 \ar[d] \ar[r] & Q_0 \ar[d] \ar[r] & 0 \\
  0 \ar[r] & K \ar[d] \ar[r] & M \ar[d] \ar[r] & N \ar[d] \ar[r] & 0, \\
           & 0               & 0               & 0 }$$$$\rm{(Fig. \;3.1)}$$
such that $P_0\rightarrow K\rightarrow 0$, $L_0\rightarrow
M\rightarrow 0$ and $Q_0\rightarrow N\rightarrow 0$ are projective
covers.
\end{lemma}
\begin{proof}
$(\Rightarrow)$ By hypothesis, $JK=K\cap JM$, which implies the
exact sequence $$\xymatrix{
  0 \ar[r] & JK \ar[r] & JM\ar[r]  & JN \ar[r] & 0.}$$ Now consider the following diagram with exact
rows and columns
$$\xymatrix{
 & 0 \ar[d] & 0 \ar[d] & 0\ar[d] \\
  0  \ar[r] & JK \ar[d] \ar[r] & JM \ar[d] \ar[r] & JN \ar[d] \ar[r] & 0 \\
  0 \ar[r] & K \ar[r] & M \ar[r] & N \ar[r] & 0,  }$$$$\rm{(Fig. \;3.2)}$$
by the ``Snake-Lemma'', we obtain the following exact sequence
$$\xymatrix{
  0 \ar[r] & K/JK \ar[r] & M/JM\ar[r]  & N/JN \ar[r] & 0.}$$ Note that for any finitely generated $R$-module $X$,
$R\otimes_{R/J}X/JX\longrightarrow X\longrightarrow 0$ is a
projective cover and if a module has projective covers then all
projective covers are unique up to isomorphisms. Now setting
$$P_0:=R\otimes_{R/J}K/JK, \;\;L_0:=R\otimes_{R/J}M/JM\;\;{\rm
and}\;\;Q_0:=R\otimes_{R/J}N/JN,$$ we have the following exact
sequence $$\xymatrix{
  0 \ar[r] & P_0 \ar[r] & L_0\ar[r]  & Q_0 \ar[r] & 0}$$
since $R/J$ is a semisimple algebra. Therefore, we have the
following commutative diagram
$$\xymatrix{
           & 0 \ar[d]        & 0 \ar[d]        & 0 \ar[d] \\
& \Omega^1(K) \ar[d] \ar[r] & \Omega^1(M) \ar[d]
  \ar[r] & \Omega^1(N) \ar[d]  \\
  0 \ar[r] & P_0 \ar[d] \ar[r] & L_0 \ar[d] \ar[r] & Q_0 \ar[d] \ar[r] & 0 \\
  0 \ar[r] & K \ar[d] \ar[r] & M \ar[d] \ar[r] & N \ar[d] \ar[r] & 0, \\
           & 0               & 0               & 0 }$$$$\rm{(Fig. \;3.3)}$$
which implies the desired diagram (Fig. 3.1) since the
``$3\times3$-Lemma''.

$(\Leftarrow)$ Suppose that we have Fig. 3.1. We may assume that
$$P_0:=R\otimes_{R/J}K/JK, \;\;L_0:=R\otimes_{R/J}M/JM\;\;{\rm
and}\;\;Q_0:=R\otimes_{R/J}N/JN$$ since the projective cover of a
module is unique up to isomorphisms. From the middle row of Fig.
3.1, we have the following exact sequence
$$\xymatrix{
  0 \ar[r] & R\otimes_{R/J}K/JK \ar[r] & R\otimes_{R/J}M/JM\ar[r]  & R\otimes_{R/J}N/JN \ar[r] & 0.}$$
Thus, we have the following short exact sequence as $R/J$-modules
$$\xymatrix{
  0 \ar[r] & K/JK \ar[r] & M/JM\ar[r]  & N/JN \ar[r] & 0}$$since $R/J$ is semisimple.
Now consider the following commutative diagram with exact rows and
columns
$$\xymatrix{
0 \ar[r] &K \ar[d] \ar[r] & M \ar[d] \ar[r] & N \ar[d] \ar[r] & 0 \\
  0 \ar[r] &K/JK \ar[d] \ar[r] & M/JM \ar[d] \ar[r] & N/JN \ar[d] \ar[r] & 0. \\
 & 0  & 0 & 0 }$$$$\rm{(Fig. \;3.4)}$$
By the ``Snake-Lemma'' again, we have the exact sequence

$$\xymatrix{
  0 \ar[r] & JK \ar[r] & JM\ar[r]  & JN \ar[r] & 0,}$$
which is equivalent to  $JK=K\cap JM$.
\end{proof}

\begin{lemma}\label{lem2}
Let $\xymatrix{
  0 \ar[r] & K\ar[r] & M\ar[r]  & N \ar[r] & 0}$ be a short exact sequence of finitely generated $R$-modules. Then $J\Omega^i(K)=\Omega^i(K)\cap
 J\Omega^i(M)$ for all $i\geq 0$ if and only if the minimal Horseshoe Lemma
holds.
\end{lemma}
\begin{proof}
$(\Rightarrow)$ By Lemma \ref{lem1}, $J\Omega^i(K)=\Omega^i(K)\cap
 J\Omega^i(M)$ for all $i\geq 0$ if and only if for all $i\geq 0$, we have the
following commutative diagram with exact rows and columns

$$\xymatrix{
           & 0 \ar[d]        & 0 \ar[d]        & 0 \ar[d] \\
  0 \ar[r] & \Omega^{i+1}(K) \ar[d] \ar[r] & \Omega^{i+1}(M) \ar[d]
  \ar[r] & \Omega^{i+1}(N) \ar[d] \ar[r] & 0 \\
  0 \ar[r] & P_i \ar[d] \ar[r] & L_i \ar[d] \ar[r] & Q_i \ar[d] \ar[r] & 0 \\
  0 \ar[r] & \Omega^i(K) \ar[d] \ar[r] & \Omega^i(M) \ar[d] \ar[r] & \Omega^i(N) \ar[d] \ar[r] & 0, \\
           & 0               & 0               & 0 }$$$$\rm{(Fig. \;3.5)}$$
such that  $P_i$, $L_i$ and $Q_i$ are projective covers of
$\Omega^i(K)$, $\Omega^i(M)$ and $\Omega^i(N)$, respectively. Now
putting these commutative diagrams together, we obtain the
commutative diagram (Fig. 1.2), i.e., the minimal Horseshoe Lemma
holds.

$(\Leftarrow)$ Suppose that the minimal Horseshoe Lemma is true for
the exact sequence
$$\xymatrix{
  0 \ar[r] & K\ar[r] & M\ar[r]  & N \ar[r] & 0,}$$i.e., we have the
commutative diagram (Fig. 1.2). Then Fig. 1.2 can be divided into a
lot of commutative diagrams similar to Fig. 3.5. Now by Lemma
\ref{lem1}, we get the desired equations.
\end{proof}

\begin{lemma}\label{lem3}

Let $\xymatrix{\xi:
  0 \ar[r] & K\ar[r] & M\ar[r]  & N \ar[r] & 0  }$ be an exact sequence
in $\mathcal{Q}^{\delta}(R)$. Then the following statements are
equivalent:
\begin{enumerate}
\item $JK=K\cap JM$;

\item $\xymatrix{
  0 \ar[r] & JK \ar[r] & JM\ar[r]  & JN \ar[r] &
  0}$ is exact;

\item $\xymatrix{
  0 \ar[r] & K/JK \ar[r] & M/JM\ar[r]  & N/JN \ar[r] &
  0}$ is exact;

\item $R/J\otimes_RK\rightarrow R/J\otimes_RM$ is a monomorphism;

\item the minimal Horseshoe Lemma holds with respect to $\xi$. \end{enumerate}
\end{lemma}
\begin{proof}
(1)$\Rightarrow$(2) and (2)$\Rightarrow$(3) have been proved in the
proof of Lemma \ref{lem2}.

(3)$\Rightarrow$(4) Consider the following commutative diagram:
$$\xymatrix{
  0  \ar[r]  & K/JK \ar[d]_{\cong} \ar[r] & M/JM \ar[d]^{\cong} \\
& R/J\otimes_RK \ar[r]  & R/J\otimes_RM,}$$$$\rm{(Fig.
\;3.6)}$$which implies that $R/J\otimes_RK\rightarrow R/J\otimes_RM$
is a monomorphism.

(4)$\Rightarrow$(1) Consider the following commutative diagram with
exact rows and columns:
$$\xymatrix{
  & 0 \ar[d]  & 0 \ar[d] & 0 \ar[d]  \\
0 \ar[r] & K \ar[d] \ar[r]  & M \ar[d] \ar[r] & N \ar[d] \ar[r]  & 0 \\
  0 \ar[r]  &K/JK \ar[d]_{\cong} \ar[r] & M/JM \ar[d]_{\cong} \ar[r]  & N/JN \ar[d]_{\cong} \ar[r]  & 0 \\
 0  \ar[r]  & R/J\otimes_RK \ar[d] \ar[r]  & R/J\otimes_RM \ar[d] \ar[r] & R/J\otimes_RN \ar[d] \ar[r]  & 0, \\
    & 0   &0   & 0  }$$$$\rm{(Fig. \;3.7)}$$
which implies that $JK=K\cap JM$ since the ``Five-Lemma'' and the
the following commutative diagram
$$\xymatrix{
  0  \ar[r] & JK \ar[d] \ar[r] & JM \ar[d]_{=} \ar[r] & JN \ar[d]_{=} \ar[r] & 0 \\
  0 \ar[r] & K\cap JM \ar[r] & JM \ar[r] & JN \ar[r] & 0.   }$$$$\rm{(Fig. \;3.8)}$$

(1)$\Rightarrow$(5) By Lemma \ref{lem1}, we have Fig. 3.1 since
$JK=K\cap JM$, thus we have the following commutative diagram with
exact rows
$$\xymatrix{
           & 0 \ar[d]        & 0 \ar[d]        & 0 \ar[d] \\
  0 \ar[r] & \Omega^1(K) \ar[d] \ar[r] & \Omega^1(M) \ar[d]
  \ar[r] & \Omega^1(N) \ar[d] \ar[r] & 0 \\
  0 \ar[r] & J^{\delta(1)-\delta(0)}P_0 \ar[r] & J^{\delta(1)-\delta(0)}Q_0  \ar[r] & J^{\delta(1)-\delta(0)}L_0 \ar[r] & 0}$$$$\rm{(Fig.
  \;3.9)}$$since $K$, $M$ and $N$ are quasi-$\delta$-Koszul modules. Now applying the functor $R/J\otimes_R-$ to Fig. 3.9, we have
the following commutative diagram with exact rows
$$\xymatrix{
           & 0 \ar[d]        & 0 \ar[d]        & 0 \ar[d] \\
& R/J\otimes_R\Omega^1(K) \ar[d]^{\alpha_1} \ar[r]^{\beta_1} &
R/J\otimes_R\Omega^1(M)
  \ar[d]^{\gamma_1}
  \ar[r] & R/J\otimes_R\Omega^1(N) \ar[d] \ar[r] & 0 \\
  0 \ar[r] & R/J\otimes_RJ^{\delta(1)-\delta(0)}P_0 \ar[r] & R/J\otimes_RJ^{\delta(1)-\delta(0)}Q_0  \ar[r] & R/J\otimes_RJ^{\delta(1)-\delta(0)}L_0 \ar[r] &
  0,}$$$$\rm{(Fig.
  \;3.10)}$$where $\alpha_1$ and $\gamma_1$ are monomorphisms since $K$, $M$ are in $\mathcal{Q}^{\delta}(R)$ and
  (1)$\Leftrightarrow$(4), which implies that
  $\beta_1$ is also a monomorphism induced by the commutativity of the left
square. By (1)$\Leftrightarrow$(4), we have
$J\Omega^1(K)=\Omega^1(K)\cap J\Omega^1(M)$. By Lemma \ref{lem1}
again, we have Fig. 3.5 in the case of $i=1$, which implies the
following commutative diagram with exact rows and columns
$$\xymatrix{
           & 0 \ar[d]        & 0 \ar[d]        & 0 \ar[d] \\
  0 \ar[r] & \Omega^2(K) \ar[d] \ar[r] & \Omega^2(M) \ar[d]
  \ar[r] & \Omega^2(N) \ar[d] \ar[r] & 0 \\
  0 \ar[r] & J^{\delta(2)-\delta(1)}P_1 \ar[r] & J^{\delta(2)-\delta(1)}Q_1  \ar[r] & J^{\delta(2)-\delta(1)}L_1 \ar[r] &
  0}$$$$\rm{(Fig.
  \;3.11)}$$since $K,\;M$ and $N$ are quasi-$\delta$-Koszul modules. Similar to the above, we have the
following commutative diagram with exact rows
$$\xymatrix{
           & 0 \ar[d]        & 0 \ar[d]        & 0 \ar[d] \\
& R/J\otimes_R\Omega^2(K) \ar[d]^{\alpha_2} \ar[r]^{\beta_2} &
R/J\otimes_R\Omega^2(M)
  \ar[d]^{\gamma_2}
  \ar[r] & R/J\otimes_R\Omega^2(N) \ar[d] \ar[r] & 0 \\
  0 \ar[r] & R/J\otimes_RJ^{\delta(2)-\delta(1)}P_1 \ar[r] & R/J\otimes_RJ^{\delta(2)-\delta(1)}Q_1  \ar[r] & R/J\otimes_RJ^{\delta(2)-\delta(1)}L_1 \ar[r] &
  0}$$$$\rm{(Fig.
  \;3.12)}$$and
$J\Omega^2(K)=\Omega^2(K)\cap J\Omega^2(M)$.

Now repeating the above procedures, we have
$J\Omega^n(K)=\Omega^n(K)\cap J\Omega^n(M)$ since the following
commutative diagram with exact rows

$$\xymatrix@C=0.8em{
           & 0 \ar[d]        & 0 \ar[d]        & 0 \ar[d] \\
  0 \ar[r] & R/J\otimes_R\Omega^n(K) \ar[d]^{\alpha_n} \ar[r]^{\beta_n} & R/J\otimes_R\Omega^n(M)
  \ar[d]^{\gamma_n}
  \ar[r] & R/J\otimes_R\Omega^n(N) \ar[d] \ar[r] & 0 \\
  0 \ar[r] & R/J\otimes_RJ^{\delta(n)-\delta(n-1)}P_{n-1} \ar[r] & R/J\otimes_RJ^{\delta(n)-\delta(n-1)}Q_{n-1}  \ar[r] & R/J\otimes_RJ^{\delta(n)-\delta(n-1)}L_{n-1} \ar[r] &
  0}$$$$\rm{(Fig.
  \;3.13)}$$
 for all $n\geq 3$. Now by Lemma \ref{lem2}, we finish the proof of (1)$\Rightarrow$(5).

(5)$\Rightarrow$(1) By Lemma \ref{lem2}, (5) is equivalent to
$J\Omega^i(K)=\Omega^i(K)\cap
 J\Omega^i(M)$ for all $i\geq 0$. In particular, let $i=0$, we have $JK=K\cap
 JM$.
\end{proof}

Now by Lemma \ref{lem3} and note that

$\{\rm Quasi$-$\rm Koszul \;modules\}\subseteq\{\rm Quasi$-$d$-$\rm
Koszul \;modules\} \subseteq\{\rm Quasi$-$\rm piecewise$-$\rm Koszul\\
\;modules\}\subseteq\{\rm Quasi$-$\delta$-$\rm Koszul \;modules\}$,
Theorem A and Corollary B are obvious.

\medskip
\section{Some applications of minimal Horseshoe Lemma}
In this section, we will give some applications of minimal Horseshoe
Lemma. More precisely, we will prove Theorems C and D.

\begin{lemma}\label{lem4} Let $R$ be an augmented Noetherian semiperfect algebra with
Jacobson radical $J$ and $\xymatrix{0 \ar[r] & K \ar[r] & M\ar[r] &
N \ar[r] & 0}$ be a short exact sequence in the category of finitely
generated $R$-modules with $JK=K\cap JM$. Then $M$ is projective if
and only if $K$ and $N$ are both projective.
\end{lemma}
\begin{proof}
$(\Rightarrow)$ By Lemma \ref{lem1}, we have Fig. 3.1, which implies
the following exact sequence
$$\xymatrix{0 \ar[r] & \Omega^1(K) \ar[r] & \Omega^1(M)\ar[r]  & \Omega^1(N)
\ar[r] & 0.}$$ By hypothesis, $M$ is a projective $R$-modules, thus
the projective cover of $M$ is itself. Hence we have
$\Omega^1(M)=0$. Now combining the above exact sequence, we have
$\Omega^1(N)=0$, which implies that $Q_0\cong N$ in Fig. 3.1, thus
$N$ is a projective $R$-module.

$(\Leftarrow)$ Assume that $K$ and $N$ are projective $R$-modules,
repeating the same argument as in the proof of the necessity, we
have $\Omega^1(K)=\Omega^1(N)=0$ since $K$ and $N$ are projective
$R$-modules, which implies that $\Omega^1(M)=0$ and hence $M$ is a
projective $R$-module.
\end{proof}
\begin{lemma}\label{lem5}
Let $R$ be a Noetherian semiperfect algebra with Jacobson radical
$J$ and $M$ a finitely generated $R$-module. Then the length of a
minimal projective resolution of $M$, denoted by $l$, equals to the
projective dimension of $M$, pd$(M)$.
\end{lemma}
\begin{proof}
By hypothesis, $M$ has a minimal projective resolution of length
$l$, we have pd$(M)\leq l$ since a minimal projective resolution is
in particular a projective resolution. But if there would be a
minimal resolution of $M$ of length strictly less than $l$, then we
have $\Ext^l_R(M, R/J)\cong \Tor^R_l(R/J,M)=0$, which is a
contradiction.
\end{proof}
\begin{lemma}\label{lem6}
Let $R$ be an augmented Noetherian semiperfect algebra with Jacobson
radical $J$ and $\xymatrix{
  0 \ar[r] & K \ar[r] & M\ar[r]  & N \ar[r] & 0}$be a short exact sequence in
the category of finitely generated $R$-modules. If the minimal
Horseshoe Lemma holds for $\xi$, then we have $pd(M)=\max\{pd(K),
pd(N)\}$.
\end{lemma}
\begin{proof}
By hypothesis the minimal Horseshoe Lemma holds, i.e., we have Fig.
1.2. More precisely, we obtain that
$$\xymatrix{
  \cdots \ar[r]  & P_2 \ar[r] & P_1 \ar[r] & P_0 \ar[r]  & K \ar[r] & 0,  }$$
$$\xymatrix{
  \cdots \ar[r]  &L_2 \ar[r] & L_1 \ar[r] &L_0 \ar[r]  & M \ar[r] & 0  }$$
and
$$\xymatrix{
  \cdots \ar[r]  & Q_2 \ar[r] & Q_1 \ar[r] & Q_0 \ar[r]  & N \ar[r] & 0  }$$
are minimal projective resolution of $K$, $M$ and $N$, respectively,
and $L_n=P_n\oplus Q_n$ for all $n\geq 0$.

If pd$(M)=\infty$, by Lemma \ref{lem5}, there exists an infinite
minimal graded projective resolution of $M$
$$\xymatrix{
  \cdots \ar[r]&L_n\ar[r]&\cdots \ar[r]  & L_2 \ar[r] & L_1 \ar[r] &L_0 \ar[r]  & M \ar[r] & 0.}$$
Note that we have $L_n=P_n\oplus Q_n$ for all $n\geq 0$ and the
minimal projective resolution of a module is unique up to
isomorphisms. Thus at least one of the lengths of
$$\xymatrix{
  \cdots \ar[r]&P_n\ar[r]&\cdots \ar[r]  & P_2 \ar[r] & P_1 \ar[r] &P_0 \ar[r]  & K \ar[r] & 0
  }$$ and $$\xymatrix{
  \cdots \ar[r]&Q_n\ar[r]&\cdots \ar[r]  & Q_2 \ar[r] & Q_1 \ar[r] &Q_0 \ar[r]  & N \ar[r] & 0
  }$$is infinite, which implies that pd$(M)=\max\{pd(K), pd(N)\}$.

If pd$(M)=n<\infty$, by Lemma \ref{lem5}, there exists a minimal
projective resolution of $M$ of length $n$:
$$\xymatrix{
  0 \ar[r]&L_n\ar[r]&\cdots \ar[r]  & L_2 \ar[r] & L_1 \ar[r] &L_0 \ar[r]  & M \ar[r] & 0,}$$
which implies that $K$ and $N$ possess the following minimal
projective resolutions
$$\xymatrix{
 0 \ar[r]&P_n\ar[r]&\cdots \ar[r]  & P_2 \ar[r] & P_1 \ar[r] &P_0 \ar[r]  & K \ar[r] &
 0,
  }$$ $$\xymatrix{
 0 \ar[r]&Q_n\ar[r]&\cdots \ar[r]  & Q_2 \ar[r] & Q_1 \ar[r] &Q_0 \ar[r]  & N \ar[r] & 0
  }$$ such that at least one of $P_n$ and $L_n$ isn't zero, which implies that pd$(M)=\max\{pd(K), pd(N)\}$ by Lemma \ref{lem5}.
\end{proof}

Now it is easy to see that Theorem C is immediate from Lemmas
\ref{lem4} and \ref{lem6}.

With the help of Theorem A and Lemma \ref{lem3}, we can prove
Theorem D directly.
\begin{proof}
(1) By Theorem A, we have Fig. 3.5 for all $i\geq0$, which implies
the following commutative diagram with exact rows and columns for
all $i\geq0$:

$$\xymatrix{
           & 0 \ar[d]        & 0 \ar[d]        & 0 \ar[d] \\
  0 \ar[r] & \Omega^{i+1}(K) \ar[d] \ar[r] & \Omega^{i+1}(M)
  \ar[d]
  \ar[r] & \Omega^{i+1}(N) \ar[d] \ar[r] & 0 \\
  0 \ar[r] & JP_i \ar[r] & JL_i  \ar[r] & JQ_i \ar[r] &
  0.}$$$$\rm{(Fig.
  \;4.1)}$$
Now applying the additive right functor $R/J\otimes_R-$ to Fig. 4.1,
we get the following commutative diagram with exact rows and columns
for all $i\geq0$:
$$\xymatrix{
             \\
 R/J\otimes_R\Omega^{i+1}(K) \ar[d]_{\beta_{i+1}}
\ar[r]^{\alpha_{i+1}} & R/J\otimes_R\Omega^{i+1}(M)
  \ar[d]_{\gamma_{i+1}}
  \ar[r] & R/J\otimes_R\Omega^{i+1}(N) \ar[d] \ar[r] & 0 \\
R/J\otimes_RJP_i \ar[r]^{\delta_{i+1}} & R/J\otimes_RJL_i  \ar[r] &
R/J\otimes_RJQ_i \ar[r] &
  0,}$$$$\rm{(Fig.
  \;4.2)}$$
where $\delta_{i+1}$ is a monomorphism for all $i\geq0$ since the
exact sequence
$$\xymatrix{0 \ar[r] & JP_i \ar[r] &JL_i\ar[r]  & JQ_i \ar[r] & 0}$$
is split, and $\gamma_{i+1}$ is a monomorphism for all $i\geq0$
since $M$ is a quasi-Koszul module and Lemma \ref{lem3}.

Now we claim that $\beta_{i+1}$ is a monomorphism for all $i\geq0$.
In fact, by the hypothesis, the minimal Horseshoe Lemma holds for
the given exact sequence $\xi$, by Lemma \ref{lem2}, we have
$J\Omega^i(K)=\Omega^i(K)\cap J\Omega^i(M)$ for all $i\geq 0$. By
Lemma \ref{lem3}, $\alpha_{i+1}$ is a monomorphism for all $i\geq0$,
which implies $\beta_{i+1}$ is a monomorphism for all $i\geq0$ since
the left above square is commutative. By Lemma \ref{lem3}, we have
$J\Omega^{i+1}(K)=\Omega^{i+1}(K)\cap J^2P_i$ for all $i\geq0$,
which imply that $K$ is a quasi-Koszul module.

(2) Similarly, we have Fig. 3.5 for all $i\geq0$. Since the minimal
Horseshoe Lemma is true for $\xi$, then by Lemma \ref{lem2}, we have
$J\Omega^i(K)=\Omega^i(K)\cap J\Omega^i(M)$ for all $i\geq 0$. By
Lemma \ref{lem3}, we have the following exact sequence
$$\xymatrix{0 \ar[r] & J\Omega^i(K) \ar[r] &J\Omega^i(M)\ar[r]  & J\Omega^i(N) \ar[r] & 0}$$
for all $i\geq 0$.

Now note that all the columns are projective covers, which imply the
following commutative diagram with exact rows and columns for all
$i\geq0$:

$$\xymatrix{
           & 0 \ar[d]        & 0 \ar[d]        & 0 \ar[d] \\
  0 \ar[r] & \Omega^{i+1}(K) \ar[d] \ar[r] & \Omega^{i+1}(M) \ar[d]
  \ar[r] & \Omega^{i+1}(N) \ar[d] \ar[r] & 0 \\
  0 \ar[r] & JP_i \ar[d] \ar[r] & JL_i \ar[d] \ar[r] & JQ_i \ar[d] \ar[r] & 0 \\
  0 \ar[r] & J\Omega^i(K) \ar[d] \ar[r] & J\Omega^i(M) \ar[d] \ar[r] & J\Omega^i(N) \ar[d] \ar[r] & 0, \\
           & 0               & 0               & 0 }$$$$\rm{(Fig. \;4.3)}$$

Now applying the additive right functor $R/J\otimes_R-$ to Fig. 4.3,
we get the following commutative diagram with exact rows and columns
for all $i\geq0$:
$$\xymatrix{
R/J\otimes_R\Omega^{i+1}(K) \ar[d]_{\epsilon_{i+1}}
\ar[r]^{\varepsilon_{i+1}} & R/J\otimes_R\Omega^{i+1}(M)
\ar[d]_{\zeta_{i+1}}
  \ar[r] & R/J\otimes_R\Omega^{i+1}(N) \ar[d]_{\eta_{i+1}} \ar[r] & 0 \\
R/J\otimes_RJP_i \ar[d] \ar[r]_{\theta_{i+1}} & R/J\otimes_RJL_i \ar[d] \ar[r] & R/J\otimes_RJQ_i \ar[d] \ar[r] & 0 \\
R/J\otimes_RJ\Omega^i(K) \ar[d] \ar[r]^{\vartheta_{i}} & R/J\otimes_RJ\Omega^i(M) \ar[d] \ar[r] & R/J\otimes_RJ\Omega^i(N) \ar[d] \ar[r] & 0. \\
         0               & 0               & 0 }$$$$\rm{(Fig. \;4.4)}$$
Similar to the analysis of (1), we have that $\epsilon_{i+1}$,
$\varepsilon_{i+1}$, $\zeta_{i+1}$ and $\theta_{i+1}$ are
monomorphisms for all $i\geq 0$. Note that
\begin{eqnarray*}
J\Omega^{i}(K)\cap J(J\Omega^{i}(M))&=&J\Omega^{i}(K)\cap
J^{2}\Omega^{i}(M)\\&=&J\Omega^{i}(K)\cap
J^{2}\Omega^{i}(M)\cap\Omega^{i}(K)\\&=&J\Omega^{i}(K)\cap
J^{2}\Omega^{i}(K)\\&=&J^2\Omega^{i}(K).
\end{eqnarray*}By Lemma \ref{lem3}, we have that $\vartheta_i$ is a
monomorphism for each $i\geq 0$. Now by ``$3\times3$-Lemma'' to Fig.
4.4, we have that $\eta_{i+1}$ is a monomorphism for each $i\geq 0$.
By Lemma \ref{lem3}, we have $J\Omega^{i+1}(N)=\Omega^{i+1}(N)\cap
J^2Q_i$ for all $i\geq0$, thus $N$ is a quasi-Koszul module.
\end{proof}
\medskip
{\bf Acknowledgements} The author would like to give his sincere
thanks to the referees for the careful reading and improved
suggestions, which improve the quality of the manuscript a lot.
\vspace{1cm}
\bibliographystyle{amsplain}

\end{document}